\documentclass[12pt]{article}

\RequirePackage[OT1]{fontenc}
\RequirePackage{amsthm,amsmath}
\RequirePackage[numbers]{natbib}
\RequirePackage[colorlinks,citecolor=blue,urlcolor=blue]{hyperref}

\usepackage{amsmath}
\usepackage{graphicx,psfrag,epsf}
\usepackage{enumerate}
\usepackage{natbib}
\usepackage{url} 
\usepackage{hyperref}
\usepackage{amsfonts}

\newtheorem{theorem}{Theorem}

\newtheorem{proposition}[theorem]{Proposition}

\newcommand{\blind}{1}

\begin{document}

\def\spacingset#1{\renewcommand{\baselinestretch}%
{#1}\small\normalsize} \spacingset{1}


\if1\blind
{
  \title{\bf The mean and variance in coupons required to complete a collection}
  \author{Rohit Pandey}
  \maketitle
} \fi

\if0\blind
{
  \bigskip
  \bigskip
  \bigskip
  \begin{center}
    {\LARGE\bf The mean and variance in coupons required to complete a collection}
\end{center}
  \medskip
} \fi

\begin{abstract}
This paper is about the Coupon collector's problem. There are some coupons, or baseball cards, or other plastic knick-knacks that are put into bags of chips or under soda bottles, etc. A collector starts collecting these trinkets and wants to form a complete collection of all possible ones. Every time they buy the product however, they don't know which coupon they will ``collect" until they open the product. How many coupons do they need to collect before they complete the collection? In this paper, we explore the mean and variance of this random variable, $N$ using various methods. Some of them work only for the special case with the coupons having equal probabilities of being collected, while others generalize to the case where the  coupons are collected with unequal probabilities (which is closer to a real world scenario).
\end{abstract}

\section*{Problems and expressions}
\subsection*{Problems}
There are $n$ coupons in a collection. A collector has the ability to purchase a coupon, but can't choose the coupon he purchases. Instead, the coupon is revealed to be coupon $i$ with probability $p_i=\frac 1 n$. Let $N$ be the number of coupons he'll need to collect before he has at least one coupon of each type. Let's call this random variable $N$. Now, we want to solve the following problems:

\begin{description}
\item[P1]The expected value of $N$ when the coupons have equal probabilities of being collected.
\item[P2]The expected value of $N$ when the coupons have unequal probabilities of being collected.
\item[P3]The variance of $N$ when the coupons have equal probabilities of being collected.
\item[P4]The variance of $N$ when the coupons have unequal probabilities of being collected.
\item[P5]The density function of $N$ (meaning the entire distribution) when the coupons have equal probabilities.
\item[P6]The density function of $N$ (meaning the entire distribution) when the coupons have unequal probabilities.
\end{description}

This paper will go over various solutions, some more powerful (can answer more of the above questions) than others. It's also clear that if we can solve the even numbered problems (2,4,6) we can simply substitute $p_i=\frac{1}{n}\;\; \forall i$ and solve the corresponding odd numbered problems (1,3,5) respectively. 

\subsection*{Expressions}
In this section, we provide the solutions to the problems, \textbf{P1} through \textbf{P6} and devote the rest of the paper to their derivations.

\begin{theorem}[Expression for P1]\label{thm:p1}
The expected number of coupons a collector will need to complete the collection when the probabilities of collecting each of the $n$ coupons is $\frac{1}{n}$ is: 

$$E(N) = n\sum\limits_{m=1}^{n}\frac{1}{m}$$
\end{theorem}

\begin{theorem}[Expression for P2]\label{thm:p2}
The variance in the number of coupons a collector will need to complete the collection when the probabilities of collecting each of the $n$ coupons is $\frac{1}{n}$ is: 

$$V(N) = n^2\sum\limits_{i=1}^n \frac{1}{i^2}-n\sum\limits_{k=1}^n \frac{1}{k}$$
\end{theorem}

\begin{theorem}[Expression for P3]\label{thm:p3}
The expected number of coupons a collector will need to complete the collection when the probabilities of collecting coupon $i$ is $p_i$ ($\sum\limits_{i=1}^{n}p_i=1$) is:

$$E(N)= \sum\limits_j\frac 1 p_j - \sum\limits_{i<j}\frac {1}{p_i+p_j} + \dots +(-1)^{m-1} \frac{1}{p_1+\dots+p_m}$$
\end{theorem}

\begin{theorem}[Expression for P4]\label{thm:p4}
The variance in the number of coupons a collector will need to complete the collection when the probabilities of collecting coupon $i$ is $p_i$ ($\sum\limits_{i=1}^{n}p_i=1$) is:
\begin{multline*}
V(N) = \left(\sum \frac {1} {p_j^2} -\sum_{i<j} \frac{1}{(p_i+p_j)^2}+\dots +(-1)^{n-1}\frac{1}{(p_1+\dots+p_n)^2}\right)-\\
\left(\sum \frac {1} {p_j} -\sum_{i<j} \frac{1}{(p_i+p_j)}+\dots +(-1)^{n-1}\frac{1}{(p_1+\dots+p_n)}\right)^2-\\
\left(\sum \frac {1} {p_j} -\sum_{i<j} \frac{1}{(p_i+p_j)}+\dots +(-1)^{n-1}\frac{1}{(p_1+\dots+p_n)}\right)
\end{multline*}

\end{theorem}

\section{A sum of geometric random variables}\label{geom_sum}

\subsection{Proof 1 of theorem \ref{thm:p1}}\label{proof:thm:p1}
Consider a state where the collector has already collected $m$ coupons. How many coupons does he need to collect to get to $m+1$? Let this be represented by the random variable, $N_m$. Then, if the total coupons needed is $N$, we have:

$$N = \sum\limits_{m=1}^n N_m$$

Every coupon collected from here is like a coin toss where with probability $\frac m n$, the collector hits a coupon he already has and makes no progress. With probability $\frac{n-m}{n}$, he collects a new coupon. So, this becomes a geometric random variable with $p=\frac{n-m}{n}$. We know that \href{"https://en.wikipedia.org/wiki/Geometric_distribution"}{a geometric random variable} has a mean $\frac{1}{p}$ and variance $\frac{1-p}{p^2}$. Hence, 

$$E(N_m)=\frac{n}{n-m}$$

Taking expectation of equation (1) and substituting we have:

$$E(N) = E(N_m) = \sum\limits_{m=1}^n \frac{n}{n-m}=n \sum\limits_{m=1}^n \frac{1}{n-m}$$

Substituting $m=n-m$ we get:

$$E(N) = n \sum\limits_{m=1}^n \frac{1}{m}$$

\subsection{Proof 1 of theorem \ref{thm:p3}}\label{proof:thm:p3}

Since the random variables $N_m$ are independent, the variance of their sum is equal to the sum of their variances. So, proceeding similarly to section \ref{proof:thm:p1} the variance, $V(N)$ can be calculated.

$$V(N) = n^2\sum\limits_{i=1}^n \frac{1}{i^2}-n\sum\limits_{k=1}^n \frac{1}{k}$$

\section{Maximum of minimums identity}\label{maxmins}

With this approach, we can prove theorems \ref{thm:p1} and \ref{thm:p2}

\subsection{Proof 1 of theorem \ref{thm:p3}}
Let $N_j$ be the number of coupons to be collected before we see the first coupon of type $j$ and $N$ the number of coupons until all are collected. We have:

$$N = \max_{1\leq j \leq n}N_j$$

In conjunction with the \href{"https://math.stackexchange.com/questions/2578996/relation-between-inclusion-exclusion-principle-and-maximum-minimums-identity"}{maximum of minimums identity} we get:

\begin{equation}\label{max_mins}N = \sum N_j - \sum_{1\leq j \leq k\leq n} \min N_j, N_k +  \sum_{1\leq j \leq k\leq i \leq n} \min N_j, N_k, N_i - \dots\end{equation}

and the fact that $\min_{1 \leq j \leq m} N_j$ is a geometric random variable with parameter $p=\sum\limits_{j=1}^m p_j$ lead to the result of theorem \ref{thm:p3} and from there, we can substitute $p_j=\frac 1 n \forall j$ to get the result of theorem \ref{thm:p1}

$$E(N) = n\sum\limits_{k=1}^n \frac 1 k$$

Note that it's not easy to get the variance, $V(N)$ with this approach because the terms in equation \ref{max_mins} are not independent.

\section{A recurrence}\label{recur}
With this approach, we can prove theorems \ref{thm:p1} and \ref{thm:p3}.

Consider a state where the collector has $m$ coupons in his collection. Let $T_m$ be the number of coupons needed to complete the collection. If the total coupons he needs to collect to complete the collection is $N$, we then have:

$$N = T_0$$

Now, we could observe that (the $N_m$ are the variables defined in section \ref{geom_sum}):

$$N_m = T_{m+1}-T_m$$

and summing over all $m$ (and noting that $T_n=0$) leads us to:

$$T_0 = \sum_m N_m$$

and this leads to the approach in section \ref{geom_sum} which makes the problem much easier to solve. Alternately, we can continue working with the $T_m$'s and construct a recurrence. Consider what happens when the collector has $m$ coupons and he collects one more. With probability $\frac{m}{n}$, he fails to add a new coupon and is back to where he started, making no progress. Let $I(\frac{n}{m})$ be a Bernoulli random variable with $p=\frac{n}{m}$. We then have the expression:

\begin{equation}\label{recurrence}T_m = 1+I\left(\frac{m}{n}\right)T_m'+\left(1-I\left(\frac{m}{n}\right)\right)T_{m+1}
\end{equation}

Where $T_m'$ is i.i.d with $T_m$.

\subsection{Proof 2 of theorem \ref{thm:p1}}
Taking expectation to both sides,

$$E(T_m) = 1+ \frac{m}{n}E(T_m)+\frac{n-m}{n}T_{m+1}$$

$$E(T_m)\left(1-\frac{m}{n}\right) = 1+ \left(1-\frac{m}{n}\right)T_{m+1}$$

$$E(T_m)-E(T_{m+1}) = \frac{n}{n-m}$$
As noted before, the L.H.S is simply $E(N_m)$ as defined in A1. In general we have,
$$\sum\limits_{m=k}^{n-1}E(T_m)-\sum\limits_{m=k}^{n-1}E(T_{m+1}) = \sum\limits_{m=k}^{n-1}\frac{n}{n-m}$$

Noting that $T_n=0$ we have,
$$E(T_k)=\sum\limits_{m=k}^{n-1}\frac{n}{n-m}$$
And letting $m=n-k$

$$E(T_{n-m}) = n\sum\limits_{k=1}^{m}\frac{1}{k}$$

We're interested in $T_0$, so let's substitute $m=n$ in equation (3).

$$E(T_0) = n \sum\limits_{k=1}^{n}\frac{1}{k}$$

\subsection{Proof 2 of theorem \ref{thm:p3}}
Now, let's try and find the variance, $V(N)=V(T_0)$. Let's square both sides of equation (1). To make the algebra easier, let's re-arrange and note that $I(\frac{m}{n})(1-I(\frac{m}{n}))=I(\frac{m}{n})-I(\frac{m}{n})^2=0$.

$$=>(T_m-1)^2 = I\left(\frac{m}{n}\right)^2 T_m'^2+(1+I\left(\frac{m}{n}\right)^2-2I\left(\frac{m}{n}\right))T_{m+1}^2$$

Now, note the following property of Bernoulli random variables: $I(\frac{m}{n})^2=I(\frac{m}{n})$. This means:

$$T_m^2-2T_m+1 = I\left(\frac{m}{n}\right) T_m'^2+(1-I\left(\frac{m}{n}\right))T_{m+1}^2$$

We have to be careful here to note which random variables are i.i.d. and which are identical. See \href{https://math.stackexchange.com/questions/3372532/how-to-square-equations-involving-random-variables}{here}.

Taking expectation and doing some algebra gives us,

$$\left(1-\frac{m}{n}\right)E(T_m^2)=2E(T_m)+\left(1-\frac{m}{n}\right)E(T_{m+1}^2)-1$$

$$=>E(T_m^2)-E(T_{m+1}^2)=2E(T_m)\frac{n}{n-m}-\frac{n}{n-m}$$

$$=>\sum\limits_{m=0}^{n-1}E(T_m^2)-\sum\limits_{m=0}^{n-1}E(T_{m+1}^2)=\sum\limits_{m=0}^{n-1}2E(T_m)\frac{n}{n-m}-\sum\limits_{m=0}^{n-1}\frac{n}{n-m}$$

$$=> E(T_0^2)-E(T_n^2)=\sum\limits_{m=0}^{n-1}2E(T_m)\frac{n}{n-m}-\sum\limits_{m=0}^{n-1}\frac{n}{n-m}$$

But, $T_n=0$ and from equation (3), $E(T_m)=n \sum\limits_{k=1}^{n-m}\frac 1 k$. So we get:

$$E(T_0^2) = \sum\limits_{m=0}^{n-1}2E(T_m)\frac{n}{n-m}-\sum\limits_{m=0}^{n-1}\frac{n}{n-m}$$

$$=>E(T_0^2) = 2n^2 \sum\limits_{m=0}^{n-1}\frac{1}{n-m}\sum\limits_{k=1}^{n-m}\frac{1}{k} -n\sum\limits_{m=0}^{n-1}\frac{1}{n-m}$$
Now, change variables $j=n-m$

$$=>E(T_0^2) = 2n^2 \sum\limits_{j=n}^{1}\frac{1}{j}\sum\limits_{k=1}^{j}\frac{1}{k} -n\sum\limits_{j=n}^{1}\frac{1}{j}$$

$$=>E(T_0^2) = 2n^2\sum\limits_{1 \leq k \leq j \leq n} \frac{1}{jk}-E(T_0)$$

This can be used in conjunction with the result of theorem \ref{thm:p1} to get the variance.

\begin{equation*}
V(T_0) = 2n^2\sum\limits_{1 \leq k \leq j \leq n} \frac{1}{jk}-E(T_0)-E(T_0)^2
\end{equation*}

Substituting the result of theorem \ref{thm:p1},
\begin{equation}\label{var_recurrence}
V(T_0) = 2n^2\sum\limits_{1 \leq k \leq j \leq n} \frac{1}{jk}-n\sum\limits_{i=1}^{n}\frac{1}{i}-\left(n\sum\limits_{i=1}^{n}\frac{1}{i}\right)^2
\end{equation}

Comparing equation \ref{var_recurrence} above with the result of theorem \ref{thm:p3} we get the easily verifiable identity:

$$2\sum_{1\leq j\leq k \leq n} \frac{1}{jk}=\sum\limits_{i=1}^n\frac{1}{i^2}+\left(\sum\limits_{i=1}^n\frac{1}{i}\right)^2$$

\section{Using a Poisson process to make dependence disappear}
Using the Poisson process to magically concoct independent random variables. This is the most powerful of all approaches since it's the only one that allows us to solve for both mean and variance for the coupon collector's problem for the general case of coupons having unequal probabilities (and higher moments as well). It is hence able to solve problems \textbf{P1} through \textbf{P4}.

In example 5.17 of \cite{ross}, the Coupon collector's problem is tackled for the general case where the probability of drawing coupon $j$ is given by $p_j$ and of course, $\sum\limits_j p_j = 1$. 

Now, he imagines that the collector collects the coupons in accordance to a Poisson process with rate $\lambda=1$. Furthermore, every coupon that arrives is of type $j$ with probability $p_j$.

Now, he defines $X_j$ as the first time a coupon of type $j$ is observed, if the $j$th coupon arrives in accordance to a Poisson process with rate $p_j$. We're interested in the time it takes to collect all coupons, $X$ (for now, eventually, we're interested in the number of coupons to be collected, $N$). So we get:

$$X = \max_{1\leq j \leq m}X_j$$

Note that if we denote $N_j$ as the number of coupons to be collected before the first coupon of type $j$ is seen, we also have for the number needed to collect all coupons, $N$:

$$N = \max_{1\leq j \leq m}N_j $$

This equation is less useful since the $N_j$ are not independent. It can still be used to get the mean (see section \ref{maxmins}), but trying to get the variance with this approach gets considerably more challenging due to this lack of independence of the underlying random variables (the are positively correlated).

But, the incredible fact that the $X_j$ are independent (discussion on that \href{https://math.stackexchange.com/questions/3421101/poisson-mixture-process-independence-used-to-devastating-effect-on-the-coupon-co}{here}), allows us to get:

\begin{equation}F_X(t) = P(X<t) = P(X_j<t \; \forall \; j) = \prod\limits_{j=1}^{m}(1-e^{-p_j t})\end{equation}

\subsection{Proof 2 of theorem \ref{thm:p2}}

Now, Ross uses the expression: $E(X) = \int\limits_0^\infty S_X(t)dt$, where $S_X(t)$ is the survival function to get:

$$E(X) = \int\limits_{0}^{\infty}\left(1-\prod\limits_{j=1}^{m}(1-e^{-p_j t})\right) dt $$

$$= \sum\limits_j\frac 1 p_j - \sum\limits_{i<j}\frac {1}{p_i+p_j} + \dots +(-1)^{m-1} \frac{1}{p_1+\dots+p_m}$$

and this proves the result of theorem \ref{thm:p2}.

\subsection{Proof 4 of theorem \ref{thm:p1}}

In the special case of all coupons having equal probabilities of being collected we have: $p_j=\frac{1}{n} \forall j$

Substituting in the equation above we get:

\begin{equation}\label{gen_expectn}
E(X) = \sum\limits_{k=1}^{n}(-1)^k \frac{{n \choose k}}{k}
\end{equation}

Let's solve a general version of the binomial sum in equation \ref{gen_expectn}.

\begin{proposition}
We have the following binomial sum:
$$\sum_{k=1}^n(-1)^{k-1}\frac{{n\choose k}}{k^r}=\sum_{i_1<i_2<\dots <i_r}\frac{1}{i_1 i_2 \dots i_r}$$
\end{proposition}
\begin{proof}
Using the Binomial theorem:

$$\frac{1-(1-t)^n}{t} = \sum\limits_{k=1}^n (-1)^{k-1}{{n \choose k}}{t^{k-1}}$$

Integrate both sides from $0$ to $x$.

$$\int\limits_0^x \frac{1-(1-t)^n}{t}dx = \sum\limits_{k=1}^n (-1)^{k-1}{{n \choose k}}\frac{x^{k}}{k}$$

For the LHS, let $1-t=u$

$$\int\limits_1^{1-x} \frac{1-(u)^n}{1-u}(-du) = \sum\limits_{k=1}^n (-1)^{k-1}{{n \choose k}}\frac{x^{k}}{k}$$

$$\frac{\sum\limits_{k=1}^n\frac{1-(1-x)^k}{k}}{x} = \sum\limits_{k=1}^n (-1)^{k-1} \frac{{n\choose k}}{k}x^{k-1}$$

Integrate both sides from $0$ to $1$, we get:

$$\sum\limits_{k=1}^n \frac 1 k \int\limits_0^1 \frac{1-(1-x)^k}{x} dx = \sum \frac{{n \choose k}}{k^2} (-1)^{k-1}$$

Substituting $1-x=t$ in the integral and expanding the geometric series we get:

$$\sum\limits_{k=1}^n \frac 1 k \sum\limits_{j=1}^k \frac 1 j = \sum \frac{{n \choose k}}{k^2} (-1)^{k-1} = \sum\limits_{k=1}^n\sum\limits_{j=1}^k \frac {1}{jk}$$

This can very easily be extended to $k^r$ in the denominator:
\begin{equation}\label{binom_sum}
\sum_{k=1}^n(-1)^{k-1}\frac{{n\choose k}}{k^r}=\sum_{i_1<i_2<\dots <i_r}\frac{1}{i_1 i_2 \dots i_r}
\end{equation}
\end{proof}

Substituting $r=1$ in equation \ref{binom_sum} and equation \ref{gen_expectn} we have,
$$E(X) = n\sum\limits_{k=1}^n \frac{1}{k}$$

Further, Ross shows that $E(N)=E(X)$ using the law of total expectation.

First, he notes,

$$E(X|N=n)=nE(T_i)$$

where $T_i$ are the inter-arrival times for coupon arrivals. Since these are assume to be exponential with rate 1,

$$E(X|N)=N$$

Taking expectations on both sides and using the law of total expectation we get:

$$E(X)=E(N)$$

\subsection{Proof 1 of theorem \ref{thm:p4}}
This approach can easily be extended to find $V(N)$, the variance (not covered by Ross). We can use the following expression to get $E(X^2)$:

$$E(X^2) = \int\limits_0^\infty 2tP(X>t)dt = \int\limits_0^\infty 2t\left(1-\prod\limits_{j=1}^n(1-e^{-p_j t})\right)dt$$

Using the fact that $\int\limits_0^\infty te^{-pt}=\frac{1}{p^2}$ and the same algebra as for $E(X)$ we get:

\begin{equation}\label{e_xsq_gen}
\frac{E(X^2)}{2} = \sum \frac {1} {p_j^2} -\sum_{i<j} \frac{1}{(p_i+p_j)^2}+\dots +(-1)^{n-1}\frac{1}{(p_1+\dots+p_n)^2}
\end{equation}

Equation \ref{e_xsq_gen} has given us $E(X^2)$ but remember that we're interested in finding $E(N^2)$ and from there, $V(N)$. So, we need to relate the variances of the two random variables. Using the law of total variance we get:

$$V(X)=E(V(X|N))+V(E(X|N))$$

So per equation (3) we have:

$$V(X)=E(V(X|N))+V(N)$$

Now, 

$$V(X|N)=NV(T_i)$$

And since $T_i \sim Exp(1)$, we have $V(T_i)=1$ meaning, $V(X|N)=N$.

Substituting into (2),

$$V(X)=E(N)+V(N)$$
So,
\begin{equation}\label{gen_var}
 V(N)=E(X^2)-E(N)-E(N)^2
\end{equation}

Substituting equation \ref{e_xsq_gen} and the result of theorem \ref{thm:p2} into equation \ref{gen_var} we get:

\begin{multline}\label{final_var}
V(N) = \left(\sum \frac {1} {p_j^2} -\sum_{i<j} \frac{1}{(p_i+p_j)^2}+\dots +(-1)^{n-1}\frac{1}{(p_1+\dots+p_n)^2}\right)-\\
\left(\sum \frac {1} {p_j} -\sum_{i<j} \frac{1}{(p_i+p_j)}+\dots +(-1)^{n-1}\frac{1}{(p_1+\dots+p_n)}\right)^2-\\
\left(\sum \frac {1} {p_j} -\sum_{i<j} \frac{1}{(p_i+p_j)}+\dots +(-1)^{n-1}\frac{1}{(p_1+\dots+p_n)}\right)
\end{multline}

\subsection{Proof 3 of theorem \ref{thm:p2}}
Now, let's consider the special case where all coupons have an equal probability of being selected. In other words, $p_j=\frac 1 n \; \forall \; j$.

We get: 

\begin{equation}\label{e_xsq}
\frac{E(X^2)}{2} = n^2\left(\sum\limits_{k=1}^n (-1)^{k-1}\frac{{n\choose k}}{k^2}\right)
\end{equation}


We now solve a general version of the binomial summation in equation \ref{e_xsq} above.

Using equations \ref{binom_sum} and \ref{e_xsq} we get:

\begin{equation}\label{e_xsq_soln}
E(X^2) = 2n^2\left( \sum_{j=1}^n\sum_{k=1}^j\frac{1}{jk}\right)
\end{equation}

Using equations \ref{e_xsq_soln} and \ref{gen_var}, we get the same result we got from the recurrence in section \ref{recur}, equation \ref{var_recurrence}.


\section*{Acknowledgements}
I'd like to thank mathexchange user, Simon for encouraging me to convert the Q\&A page on this into a paper.


\begin{thebibliography}{9}


\bibitem{ross}
\textsc{Ross, S.} (2010). \textit{Introduction to Probability Models}, 10th ed.
Elsevier.


\end{thebibliography}
\end{document}